\documentclass{amsart}
\usepackage{bm}

\usepackage{amsmath}
\usepackage{amsthm}
\usepackage{amssymb}
\usepackage{enumitem}

\usepackage{mathrsfs} 

\newcommand{\R}{\mathbb{R}}
\newcommand{\C}{\mathbb{C}}
\newcommand{\lr}[1]{\langle #1 \rangle}
\newcommand{\eps}{\varepsilon}
\newcommand{\norm}[1]{\| #1 \|}
\newcommand{\Norm}[1]{\left\| #1\right\|}

\newtheorem{thm}{Theorem}[section]
\newtheorem{prop}[thm]{Proposition}
\newtheorem{dfn}[thm]{Definition}
\newtheorem{lem}[thm]{Lemma}

\theoremstyle{remark}
\newtheorem{rmk}[thm]{Remark}

\DeclareMathOperator{\supp}{supp}

\allowdisplaybreaks[1]

\title[Random data Cauchy theory for 4NLS with cubic nonlinearity]{Random data Cauchy theory for the fourth order nonlinear Schr\"{o}dinger equation with cubic nonlinearity}

\author[H. Hirayama]{Hiroyuki Hirayama}
\address{Graduate School of Mathematics, Nagoya University, Chikusa-ku, Nagoya, 464-8602, Japan}
\email{m08035f@math.nagoya-u.ac.jp}

\author[M. Okamoto]{Mamoru Okamoto}
\address{Department of Mathematics, Institute of Engineering, Academic Assembly, Shinshu University, 4-17-1 Wakasato, Nagano City 380-8553, Japan}
\email{m\_okamoto@shinshu-u.ac.jp}
\date{}


\numberwithin{equation}{section}

\begin{document}

\begin{abstract}
We consider the Cauchy problem for the fourth order nonlinear Schr\"{o}dinger equation with derivative nonlinearity $(i\partial _t + \Delta ^2) u= \pm \partial (|u|^2u)$ on $\R ^d$, $d \ge 3$, with random initial data, where $\partial$ is a first order derivative with respect to the spatial variable, for example a linear combination of  $\frac{\partial}{\partial x_1} , \, \dots , \, \frac{\partial}{\partial x_d}$ or $|\nabla |= \mathcal{F}^{-1}[|\xi | \mathcal{F}]$.
We prove that almost sure local in time well-posedness, small data global in time well-posedness and scattering hold in $H^s(\R ^d)$ with $\max ( \frac{d-5}{2}, \frac{d-5}{6}) < s$, whose lower bound is below the scale critical regularity $s_c= \frac{d-3}{2}$.
\end{abstract}

\maketitle

\section{Introduction}

We consider the Cauchy problem for the fourth order nonlinear Schr\"{o}dinger equation with derivative nonlinearity:
\begin{equation}
\label{4NLS}
\left\{
\begin{aligned}
& (i \partial _t + \Delta ^2 ) u = \pm \partial ( |u|^2 u), \\
& u(0, \cdot ) = \phi .
\end{aligned}
\right.
\end{equation}
Here, $u : \R \times \R ^d \rightarrow \C$ is an unknown function, $\phi : \R ^d \rightarrow \C$ is a given function, $\partial$ is a first order derivative with respect to the spatial variable, for example a linear combination of  $\frac{\partial}{\partial x_1} , \, \dots , \, \frac{\partial}{\partial x_d}$ or $|\nabla |= \mathcal{F}^{-1}[|\xi | \mathcal{F}]$.
The fourth order Schr\"{o}dinger equation appears in the study of deep water wave dynamics \cite{Dysthe}, solitary waves \cite{Karpman}, \cite{KS}, vortex filaments \cite{Fukumoto}, and so on.
The fourth order Schr\"{o}dinger equation are widely studied, for the results for derivative nonlinearity see \cite{HJ1}, \cite{HHW}, \cite{HJ2}, \cite{Wang}, \cite{HN1}, \cite{HN2}, \cite{HO} and references cited therein.

This equation has scale invariance.
Namely, if $u$ is a solution, then $u_{\mu} (t,x) := \mu ^{-\frac{3}{2}} u(\frac{t}{\mu ^4} , \frac{x}{\mu})$ is also solution.
A simple calculation shows
\begin{equation*}
\norm{u_{\mu}(0,\cdot )}_{\dot{H}^s} = \mu ^{-s + \frac{d-3}{2}} \norm{u(0,\cdot )}_{\dot{H}^s}.
\end{equation*}
Hence, the scale critical regularity is $s_c :=\frac{d-3}{2}$.
We may expect that the scale critical exponent is a threshold between regularity to be well-posed and ill-posed.
According to expectation, the fourth order nonlinear Schr\"{o}dinger equation \eqref{4NLS} is well-posed in $H^{\frac{d-3}{2}}(\R ^d)$ if $d \ge 2$, see for example \cite{HO}.
Of cause, it is not always true and there is also a possibility of presence of a gap between the scaling prediction and the optimal well-posedness regularities (see for example \cite{CCT}).

Burq and Tzvetkov \cite{BT1}, \cite{BT2} however proved that the Cauchy problem for the cubic wave equation on the three dimensional compact Riemannian manifold $M$ is almost surely well-posed in $H^s(M)$ with $s>\frac{1}{4}$ which is deterministically ill-posed in $H^s(M)$ if $s<\frac{1}{2}$.
Recently, B\'{e}nyi, Oh, and Pocovnicu \cite{BOP1, BOP2} and L\"{u}hrmann and Mendelson \cite{LM} showed that the almost sure well-posedness for the Cauchy problem on unbounded domains $\R ^d$ in the super critical Sobolev spaces.
Because of these results, we expect that the scale critical exponent is not the threshold in the random data Cauchy problem.
We therefore consider the almost sure well-posedness of the Cauchy problem with the initial data in the super critical Sobolev spaces.

Following the papers \cite{BOP1}, \cite{BOP2}, we define the randomization.
Let $\psi \in \mathcal{S}(\R ^d)$ satisfy
\[
\supp \psi \subset [-1,1]^d , \quad \sum _{n \in \mathbb{Z}^d} \psi (\xi -n) =1 \quad \text{for any $\xi \in \R ^d$}.
\]
Let $\{ g_n \}$ be a sequence of independent mean zero complex valued random variables on a probability space $(\Omega , \mathcal{F} ,P)$, where the real an imaginary parts of $g_n$ are independent and endowed with probability distributions $\mu _n^{(1)}$ and $\mu _n^{(2)}$.
Throughout this paper, we assume that there exists $c>0$ such that
\[
\left| \int _{\R} e^{\kappa x} d \mu _{n}^{(j)} (x) \right| \le e^{c \kappa ^2}
\]
for all $\kappa \in \R$, $n \in \mathbb{Z}^d$, $j=1,2$.
This condition is satisfied by the standard complex valued Gaussian random variables and the standard Bernoulli random variables.
We then define the Wiener randomization of $\phi$ by
\begin{equation} \label{eq:rand}
\phi ^{\omega} := \sum _{n \in \mathbb{Z}^d} g_n (\omega ) \psi (D-n) \phi .
\end{equation}

The randomization has no smoothing in terms of differentiability (\cite[Appendix B]{BT1}).
However, it improves the integrability (see for example Lemma 2.3 in \cite{BOP1}).
From this point of view, the randomization makes the problem subcritical in some sense.

\begin{thm} \label{thm:LWP}
Let $d \ge 3$ and $\max ( \frac{d-5}{2}, \frac{d-5}{6}) <s< \frac{d-3}{2}$.
Given $\phi \in H^s(\R ^d)$, let $\phi ^{\omega}$ be its randomization defined by \eqref{eq:rand}.
Then, for almost all $\omega \in \Omega$, there exist $T_{\omega} >0$ and a unique solution $u$ to \eqref{4NLS} with $u(0,x) = \phi ^{\omega}(x)$ in a space continuously embedded in
\[
S(t) \phi ^{\omega} + C((-T_{\omega},T_{\omega}) ; H^{\frac{d-3}{2}}(\R ^d)) \subset C((-T_{\omega},T_{\omega}) ; H^{s}(\R ^d)).
\]
More precisely, there exist $C, c>0$, $\gamma >0$ such that for each $0<T<1$, there exists $\Omega _{T} \subset \Omega$ with $P(\Omega _T) \ge 1- C \exp \left( - \frac{c}{T^{\gamma} \norm{\phi}_{H^s}^2} \right)$.
\end{thm}

Since well-posedness in $H^s(\R ^d)$ with $s \ge s_c = \frac{d-3}{2}$ holds in the deterministic setting, we concentrate in the case $s < \frac{d-3}{2}$.
Theorem \ref{thm:LWP} says that almost sure local in time well-posedness for \eqref{4NLS}.
Namely, \eqref{4NLS} possesses local strong solutions for a large class of functions in $H^s(\R ^d)$ with $s<\frac{d-3}{2}$.

Decomposing $u$ into the linear and nonlinear parts, we estimate the contributions for each part to the nonlinearity.
Thanks to the Strichartz estimates, the nonlinear part has more regularity than the initial data.
More precisely, the nonlinear part belongs to $C((-T_{\omega}, T_{\omega}) ; H^{\frac{d-3}{2}}(\R ^d))$ even if $\phi \in H^s(\R )$ with $s<\frac{d-3}{2}$.
On the other hand, by the randomization, the improved Strichartz estimate holds for the linear part with randomized initial data (see Lemma \ref{lem:stcStr} below), but it remains $C((-T_{\omega}, T_{\omega}) ;H^s(\R ^d))$.

Next, we focus on the global behavior of the solution.

\begin{thm} \label{thm:WP}
Let $d \ge 3$ and $\max ( \frac{d-5}{2}, \frac{d-5}{6}) <s< \frac{d-3}{2}$.
Given $\phi \in H^s(\R ^d)$, let $\phi^{\omega}$ be its randomization defined by \eqref{eq:rand}.
Then, there exist $C, c>0$ and $\Omega _{\phi} \subset \Omega$ such that with the following properties:
\begin{enumerate}[label=(\alph*)]
\item \label{gwp1} $P(\Omega _{\phi}) \ge 1- C \exp \left( - \frac{c}{\norm{\phi}_{H^s}^2} \right)$.
\item \label{gwp2} For each $\omega \in \Omega$, there exists a (unique) global in time solution $u$ to \eqref{4NLS} with $u(0,x) = \phi ^{\omega} (x)$ in the class
\[
S(t) \phi ^{\omega} + C(\R ; H^{\frac{d-3}{2}}(\R ^d)) \subset C(\R ; H^{s}(\R ^d)).
\]
\item \label{gwp3} For each $\omega \in \Omega _{\phi}$, there exists $v_{\pm}^{\omega} \in H^{\frac{d-3}{2}}(\R ^d)$ such that
\[
\norm{u(t) - S(t) (\phi ^{\omega} + v_{\pm}^{\omega})}_{H^{\frac{d-3}{2}}} \rightarrow 0
\]
as $t \rightarrow \pm \infty$.
\end{enumerate}
\end{thm}

The uniqueness holds in the space $X^s$ defined by \eqref{functsp} below, which is a subspace continuously embedded in $S(t) \phi ^{\omega} + C(\R ; H^{\frac{d-3}{2}}(\R ^d))$

Theorem \ref{thm:WP} says that the probabilistic small data global well-posedness and scattering because \ref{gwp1} is meaningful if
\[
\norm{\phi}_{H^s} \le \left( \frac{c}{\log C} \right) ^{\frac{1}{2}}.
\]
Moreover, it follows from Theorem \ref{thm:WP} that for almost all $\omega \in \Omega$, there exists  $\eps (\omega ) >0$ such that for every $\eps \in (0, \eps (\omega ))$, there exists a global in time solution $u$ to \eqref{4NLS} with $u(0,x) = \eps \phi ^{\omega}(x)$ in a space continuously embedded in $C(\R ; H^{s}(\R ^d))$ (see Remark \ref{almostsureGWP}).

\begin{rmk}
The one and two dimensional cases are not included in Theorems \ref{thm:LWP} and \ref{thm:WP}.
We believe that it is hard to prove the same results in $d=1,2$ because of the following difficulties.
The worst interaction, which is so called ``resonance,'' occurs, because we focus on the nonlinearity $\partial (|u|^2u)$.
Accordingly, to recover derivative, we only rely on the Strichartz estimates.
In other words, we can not expect some good structure of the nonlinearity.

Since the randomization makes the integrability better, but the smoothness the same, from the randomized Strichartz estimate, it only recovers the  one half derivative in $d=1$ and one derivative in $d \ge 2$, which is the same as the deterministic case.
To consider the nonlinear estimate, we can therefore recover at most one derivative in $d=1$ and two derivatives in $d \ge 2$ even in the randomized case.
It however does not enough to get almost sure well-posedness in $H^s(\R )$ with $s<0$ and $H^s(\R ^2)$ with $s<-\frac{1}{2}$ because of the presence of one derivative in front of the nonlinear part.
Furthermore, the lower bound $\frac{d-5}{2} = s_c-1$ for $d \ge 5$ is almost optimal from this point of view.
\end{rmk}

From a scaling argument as in \cite{BOP2}, we may consider the global well-posedness for large initial data in a large probability.
Given $\mu >0$, define $\psi ^{\mu}$ by
\[
\psi ^{\mu} (\xi ) = \psi \left( \frac{\xi}{\mu} \right) .
\]
Then, the following decomposition holds: for any function $\phi$ on $\R ^d$,
\[
\phi =\sum _{n \in \mathbb{Z}^d} \psi ^{\mu}( D-  \mu n) \phi .
\]
We define the randomization of $\phi ^{\omega, \mu}$ of $\phi$ on dilated cubes of scale $\mu$ by
\begin{equation} \label{randdil}
\phi ^{\omega ,\mu} := \sum _{n \in \mathbb{Z}^d} g_n(\omega ) \psi ^{\mu} ( D- \mu n) \phi .
\end{equation}
Then, we have the following global in time well-posedness of \eqref{4NLS} for large initial data with a large probability.

\begin{thm} \label{thm:GWP}
Let $d \ge 4$ and $\phi \in H^s(\R ^d)$ for $\max ( \frac{d-5}{2}, \frac{d-5}{6}) <s< \frac{d-3}{2}$.
Then, for each $0< \eps \ll 1$, there exists a small dilation scale $\mu _0 = \mu _0 (\eps , \norm{\phi}_{H^s})>0$ such that for every $\mu \in (0,\mu _0)$, there exists a set $\Omega _{\mu} \subset \Omega$ with the following properties:
\begin{enumerate}[label=(\alph*)]
\item \label{gwp-l1} $P(\Omega _{\mu} ) \ge 1- \eps$.
\item \label{gwp-l2} If $\phi ^{\omega ,\mu}$ is the randomization define in \eqref{randdil}, then for each $\omega \in \Omega$, there exists a (unique) global in time solution to \eqref{4NLS} with $u(0,x) = \phi^{\omega , \mu}(x)$ in the class
\[
S(t) \phi^{\omega , \mu}(x) + C(\R ; H^{\frac{d-3}{2}}(\R ^d)) \subset C(\R ; H^{s}(\R ^d)).
\]
\item\label{gwp-l3} For each $\omega \in \Omega _{\mu}$, there exists $v_{\pm}^{\omega} \in H^{\frac{d-3}{2}}(\R ^d)$ such that
\[
\norm{u(t) - S(t) (\phi^{\omega , \mu} + v_{\pm}^{\omega})}_{H^{\frac{d-3}{2}}} \rightarrow 0
\]
as $t \rightarrow \pm \infty$.
\end{enumerate}
\end{thm}

Since we have to extract a measurable set with positive measure because of \ref{gwp-l1}, Theorem \ref{thm:GWP} is not an almost sure global well-posednss result.

We observe that thanks to $s >0$
\begin{equation} \label{scale}
\norm{u_{\mu}(0, \cdot )}_{H^s}=\mu^{\frac{d-3}{2}}\norm{\langle \mu^{-1}\xi \rangle^{s}\widehat{u}(0,\xi )}_{L^{2}_{\xi}}
\leq \mu^{\frac{d-3}{2}-\max (s,0)}\norm{u (0, \cdot )}_{H^{s}}
\end{equation}
for  $0<\mu <1$.
Roughly speaking, by the scaling which makes the initial data small enough for $\max\{s,0\}<\frac{d-3}{2}$, Theorem \ref{thm:GWP} is reduced to the small data setting.
To modify randomization as in \eqref{randdil}, we can treat the random data Cauchy problem.
In order to employ the scaling argument, we assume $d \ge 4$ because $d=3$ is the mass critical.
Indeed, by the scaling argument \eqref{scale}, we can not make the norm $\norm{\phi}_{H^s}$ small if $d=3$.

We now give a brief outline of this article.
In Section \ref{S:2}, we collect lemmas which are used in the proof of our main results.
In Section \ref{S:3}, we define the function spaces and show these properties and prove the main nonlinear estimates.
In Section \ref{S:4}, we give a proof of almost sure well-posedness results, Theorems \ref{thm:LWP}, \ref{thm:WP}.
In Section \ref{S:5}, we prove Theorem \ref{thm:GWP}.

\section{The probabilistic lemmas} \label{S:2}

Firstly, we recall the probabilistic estimate.

\begin{lem}[\cite{BT1}] \label{lem:ranint}
There exists $C>0$ such that
\[
\Norm{\sum _{n \in \mathbb{Z}^d} g_n(\omega ) c_n}_{L^p(\Omega )} \le C \sqrt{p} \norm{c_n}_{l^2}
\]
for all $p \ge 2$ and $\{ c_n \} \in l^2$.
\end{lem}

The randomization keeps differentiability of the function.

\begin{lem}[\cite{BOP1}] \label{lem:randiff}
Given $\phi \in H^s (\R ^d)$, let $\phi ^{\omega}$ be its randomization defined by \eqref{eq:rand}.
Then, there exist $C,c >0$ such that
\[
P( \norm{\phi ^{\omega}}_{H^s} > \lambda ) < C \exp \left( - c \frac{\lambda ^2}{\norm{\phi}_{H^s}^2} \right)
\]
for all $\lambda >0$.
\end{lem}

Let $S(t) := e^{it \Delta ^2}$ be the linear propagator of the fourth order Schr\"{o}dinger group,
Namely, $v(t) = S(t) \phi$ solves
\[
(i\partial _t + \Delta ^2) v= 0, \quad v(0) = \phi.
\]
We say that a pair $(q,r)$ is Schr\"{o}dinger admissible if $2 \le q,r \le \infty$, $(q,r,d) \neq (2, \infty ,2)$, and
\[
\frac{2}{q} + \frac{d}{r} = \frac{d}{2}.
\]
We say that a pair $(q,r)$ is biharmonic admissible if $2 \le q ,r \le \infty$, $(q,r,d) \neq (2, \infty ,4)$, and
\[
\frac{4}{q} + \frac{d}{r} = \frac{d}{2}.
\]
The following Strichartz estimates hold.

\begin{prop}[\cite{Pau}] \label{prop:Str}
\begin{enumerate}
\item Let $(q,r)$ be biharmonic admissible.
Then, we have
\[
\norm{S(t) \phi}_{L_t^q L_x^r (\R \times \R ^d)} \lesssim \norm{\phi}_{L^2}.
\]
Let $(\tilde{q}, \tilde{r})$ be also biharmonic admissible and $(\tilde{q}', \tilde{r}')$ be the pair of conjugate exponents of $(\tilde{q}, \tilde{r})$.
Then, we have
\[
\Norm{\int _{\R} S(-t')F(t') dt' }_{L_x^2(\R ^d)} + \Norm{\int _0^t S(t-t')F(t') dt' }_{L_t^q L_x^r (\R \times \R ^d)} \lesssim \norm{F}_{L_t^{\tilde{q}'} L_x^{\tilde{r}'}(\R \times \R ^d)}.
\]

\item Let $(q,r)$ be Schr\"{o}dinger admissible.
Then, we have
\[
\norm{S(t) \phi}_{L_t^q L_x^r (\R \times \R ^d)} \lesssim \norm{|\nabla |^{-\frac{2}{q}} \phi}_{L^2}.
\]
Let $(\tilde{q}, \tilde{r})$ be also Schr\"{o}dinger admissible and $(\tilde{q}', \tilde{r}')$ be the pair of conjugate exponents of $(\tilde{q}, \tilde{r})$.
Then, we have
\begin{align*}
& \Norm{\int _{\R ^d} S(-t')F(t') dt' }_{L_x^2 (\R ^d)} \lesssim \norm{|\nabla |^{-\frac{2}{\tilde{q}}} F}_{L_t^{\tilde{q}'} L_x^{\tilde{r}'}(\R \times \R ^d)} , \\
& \Norm{\int _0^t S(t-t')F(t') dt' }_{L_t^q L_x^r (\R \times \R ^d)} \lesssim \norm{|\nabla |^{-\frac{2}{q}-\frac{2}{\tilde{q}}} F}_{L_t^{\tilde{q}'} L_x^{\tilde{r}'}(\R \times \R ^d)}.
\end{align*}
\end{enumerate}
\end{prop}

By the randomization, improved Strichartz type estimates hold.

\begin{lem} \label{lem:stcStr}
Given $\phi \in L^2(\R ^d)$, let $\phi ^{\omega}$ be its randomization defined by \eqref{eq:rand}.

\begin{enumerate}[label=(\roman*)]
\item Let $(q,r)$ be biharmonic admissible with $q,r<\infty$ and $r \le \bar{r} <\infty$.
Then, there exist $C, c>0$ such that
\[
P(\norm{S(t) \phi ^{\omega}}_{L_t^q L_x^{\bar{r}} (\R \times \R ^d)} > \lambda ) \le C \exp \left( -c \frac{\lambda ^2}{\norm{\phi}_{L^2}^2} \right)
\]
for all $\lambda >0$.

\item Let $(q,r)$ be Schr\"{o}dinger admissible with $q,r<\infty$ and $r \le \bar{r} < \infty$.
Then, there exist $C, c>0$ such that
\[
P(\norm{|\nabla |^{\frac{2}{q}} S(t) \phi ^{\omega}}_{L_t^q L_x^{\bar{r}} (\R \times \R ^d)} > \lambda ) \le C \exp \left( -c \frac{\lambda ^2}{\norm{\phi}_{L^2}^2} \right)
\]
for all $\lambda >0$.
\end{enumerate}
\end{lem}

\begin{proof}
By the same argument as in \cite{BT1} and \cite{BOP1}, it suffices to show that
\[
\left( \int _{\Omega} \norm{S(t) \phi ^{\omega}}_{L_t^q L_x^{\bar{r}}} ^p dP \right) ^{1/p} \lesssim \sqrt{p} \norm{\phi}_{L^2}.
\]
for $p \ge \max ( q, \tilde{r})$.
By the Minkowski and Bernstein inequalities, Lemma \ref{lem:ranint}, and Proposition \ref{prop:Str}, we have
\begin{align*}
\left( \int _{\Omega} \norm{S(t) \phi ^{\omega}}_{L_t^q L_x^{\bar{r}}} ^p dP \right) ^{1/p} 
& \lesssim \Norm{ \norm{S(t) \phi ^{\omega}}_{L^p(\Omega )} }_{L_t^q L_x^{\bar{r}}}
\lesssim \sqrt{p} \Norm{ \norm{\psi (D-n) S(t) \phi}_{l^2}}_{L_t^q L_x^{\bar{r}}} \\
& \lesssim \sqrt{p} \Norm{ \norm{\psi (D-n) S(t) \phi}_{L_x^{\bar{r}}}}_{L_t^q l^2}
\lesssim  \sqrt{p} \Norm{ \norm{\psi (D-n) S(t) \phi}_{L_t^q L_x^{r}}}_{l^2} \\
& \lesssim \sqrt{p} \Norm{ \norm{\psi (D-n) S(t) \phi}_{L_x^2}}_{l^2}
\sim \sqrt{p} \norm{\phi}_{L^2}.
\end{align*}

Since the case (b) is similarly handled, we omit the details here.
\end{proof}

\section{Nonlinear estimates} \label{S:3}

Before showing the probabilistic nonlinear estimates, we observe some corollaries of the Strichartz estimates.

\begin{dfn} \label{functionsp}
We define $\mathcal{X}$ to be the space of the completion of $\mathcal{S}(\R \times \R ^{d})$ with respect to the following norm:
\begin{align*}
\norm{u}_{\mathcal{X}} & := \norm{u}_{L_t^{\infty} L_x^2} + \norm{|\nabla | u}_{L_t^2 L_x^{\frac{2d}{d-2}}} + \norm{|\nabla |^{\frac{1}{2}} u}_{L_t^4 L_x^{\frac{2d}{d-1}}} + \norm{|\nabla |^{2\delta}u}_{L_t^{\frac{1}{\delta}} L_x^{\frac{2d}{d-4\delta}}},
\end{align*}
where $\delta >0$ is a small constant.
We also define the space $\mathcal{Y}$ to be the space dual to $\mathcal{X}$ with appropriate norm. 
In other words, the space $\mathcal{Y}$ is the sum of Banach spaces $L_t^1 L_x^2$, $L_t^2 |\nabla | L_x^{\frac{2d}{d+2}}$, $L_t^{\frac{4}{3}} |\nabla |^{\frac{1}{2}}L_x^{\frac{2d}{d+1}}$, and $L_t^{\frac{1}{1-\delta}} |\nabla |^{2\delta} L_x^{\frac{2d}{d+4\delta}}$ whose norm is defined by the usual manner.
Let $s \in \R$.
We define $X^s$ to be the space of completion of $\mathcal{S}(\R \times \R ^d)$ with respect to the norm
\begin{equation} \label{functsp}
\norm{u}_{X^s} := \norm{\lr{\nabla}^s u}_{\mathcal{X}}
\end{equation}
\end{dfn}

We note that $(2, \frac{2d}{d-2})$, $(4, \frac{2d}{d-1})$ and $(\frac{1}{\delta}, \frac{2d}{d-4\delta})$ are Schr\"{o}dinger admissible pairs if $d \ge 3$.
By Proposition \ref{prop:Str}, we have
\begin{equation} \label{naive}
\norm{S(t) \phi}_{\mathcal{X}} \lesssim \norm{\phi}_{L^2}, \quad
\Norm{\int _0^t S(t-t') F(t') dt'}_{\mathcal{X}} \lesssim \norm{F}_{\mathcal{Y}} .
\end{equation}

Given an interval $I \subset \R$, we define by $X^s(I)$ the time restricted space of $X^s$.
Namely,
\[
\norm{u}_{X^s(I)} := \inf \{ \norm{v}_{X^s} : u=v \text{ on } I \times \R ^d \} .
\]

Put $\mathcal{N} (u) = \partial (|u|^2u)$.
Let $z(t) = z^{\omega} (t) := S(t) \phi ^{\omega}$ and $v(t) = u(t) - z(t)$ be the linear and nonlinear parts of $u$ respectively.
Then, \eqref{4NLS} is equivalent to the following perturbed equation:
\begin{equation} \label{eq:p4NLS}
\left\{
\begin{aligned}
& (i\partial _t + \Delta ^2) v = \pm \mathcal{N}(v+z) , \\
& v(0,x) =0
\end{aligned}
\right.
\end{equation}
Define $\Gamma$ by
\[
\Gamma v(t) := \mp i \int _0^t S(t-t') \mathcal{N}(v+z) (t') dt'.
\]
To state probabilistic nonlinear estimates, we define the following norms:
For an interval $I \subset \R$,
\begin{gather} \label{norm:S_0}
\begin{aligned}
\norm{u}_{S_0(I)} & := \max \{ \norm{u}_{L_t^q L_x^r (I \times \R ^d)} : (q,r) = (\tfrac{1}{\delta}, \tfrac{2d}{d-8\delta}) , (\tfrac{1}{\delta}, d), (\tfrac{1}{\delta}, \tfrac{d}{1-2\delta}), (\tfrac{1}{\delta}, \tfrac{4d}{d-1+8\delta}) \} \\
& \qquad + \max \{ \norm{|\nabla |^{\frac{2}{q}} u}_{L_t^q L_x^r (I \times \R ^d)} :  (q,r) = (4, 2d), (4, \tfrac{4d}{1+16\delta}), (\tfrac{2}{1-2\delta}, \tfrac{d}{4\delta}), (\tfrac{2}{1-2\delta}, 4d),  \\
& \hspace*{100pt} (\tfrac{2}{1-2\delta}, \tfrac{4d}{d-3+16\delta}) , (\tfrac{2}{1-4\delta}, \tfrac{2d}{d-2}),  (\tfrac{2}{1-4\delta}, \tfrac{d}{8\delta}),  (\tfrac{2}{1-4\delta}, \tfrac{2d}{d-4+16\delta}),  (\tfrac{2}{1-4\delta}, \tfrac{4d}{d-1+8\delta}) \} ,
\end{aligned}
\\
\label{norm:S_0'}
\begin{aligned}
\norm{u}_{S_0'(I)} & := \max \{ \norm{u}_{L_t^q L_x^r (I \times \R ^d)} :  (q,r) =(\tfrac{1}{\delta}, \tfrac{2d}{d-8\delta}) , (\tfrac{2}{\delta}, d), (\tfrac{2}{\delta}, \tfrac{d}{1-2\delta}), (\tfrac{2}{\delta}, \tfrac{4d}{d-1+8\delta}) \} \\
& \qquad + \max \{ \norm{|\nabla |^{\frac{2}{q}} u}_{L_t^q L_x^r (I \times \R ^d)} : (q,r) = (\tfrac{4}{1-\delta}, 2d), (\tfrac{4}{1-\delta}, \tfrac{4d}{1+16\delta}), (\tfrac{2}{1-3\delta}, \tfrac{d}{4\delta}), (\tfrac{2}{1-3\delta}, 4d),\\
& \hspace*{100pt}  (\tfrac{2}{1-3\delta}, \tfrac{4d}{d-3+16\delta}) , (\tfrac{2}{1-5\delta}, \tfrac{2d}{d-2}),  (\tfrac{2}{1-5\delta}, \tfrac{d}{8\delta}),  (\tfrac{2}{1-5\delta}, \tfrac{2d}{d-4+16\delta}),  (\tfrac{2}{1-5\delta}, \tfrac{4d}{d-1+8\delta}) \} .
\end{aligned}
\end{gather}
The following is the main result in this section.

\begin{lem} \label{lem:nonest}
Let $d \ge 3$, $\max ( \frac{d-5}{2}, \frac{d-5}{6}) <s< \frac{d-3}{2}$, and $\delta >0$ be sufficiently small depending only on $d$ and $s$.
Given $\phi \in H^s(\R ^d)$, let $\phi ^{\omega}$ be its randomization defined by \eqref{eq:rand}.
For $R>0$, we put
\[
E_R := \{ \omega \in \Omega : \norm{\phi ^{\omega}}_{H^s} + \norm{\lr{\nabla}^s S(t) \phi ^{\omega}}_{S_0(\R )} \le R \} .
\]
Then, we have
\begin{align}
\norm{\Gamma v}_{X^{\frac{d-3}{2}}} & \le C_1 \left( \norm{v}_{X^{\frac{d-3}{2}}}^3 + R ^3 \right) , \label{eq:nonest1} \\
\norm{\Gamma v_1- \Gamma v_2}_{X^{\frac{d-3}{2}}} & \le C_2 \left( \norm{v_1}_{X^{\frac{d-3}{2}}}^2 + \norm{v_2}_{X^{\frac{d-3}{2}}}^2 + R^2 \right) \norm{v_1-v_2}_{X^{\frac{d-3}{2}}} \label{eq:nonest2}
\end{align}
for all $v, v_1, v_2 \in X^{\frac{d-3}{2}}$ and $\omega \in E_R$.
\end{lem}

\begin{proof}
We only prove \eqref{eq:nonest1} because \eqref{eq:nonest2} follows from in a similar manner.
Proposition \ref{prop:Str} yields that
\[
\norm{\Gamma v}_{X^{\frac{d-3}{2}}}
\lesssim \norm{\lr{\nabla}^{\frac{d-3}{2}} \partial (w_1 w_2 w_3)}_{\mathcal{Y}}
\]
where $w_j= v$ or $z$ ($j=1,2,3$).

We firstly note that the following facts:
\begin{enumerate}[label=(f\arabic*)]
\item \label{d1} $(4, \frac{2d}{d-1})$ is  Schr\"{o}dinger admissible and its dual pair is $(\frac{4}{3}, \frac{2d}{d+1})$.
\item \label{f1} $(\frac{1}{\delta}, \frac{2d}{d-8\delta})$ is biharmonic admissible and its dual pair is $(\frac{1}{1-\delta}, \frac{2d}{d+8\delta})$.
\item \label{f2} $(\frac{2}{1-2\delta}, \frac{2d}{d-2+4 \delta})$ and $(\frac{2}{1-4\delta}, \frac{2d}{d-2+8 \delta})$ are Schr\"{o}dinger admissible.
\item \label{f3} $(2, \frac{2d}{d-2})$ is Schr\"{o}dinger admissible and its dual pair is $(2, \frac{2d}{d+2})$.
\end{enumerate}
We use the dyadic decomposition except for the Case 1 below:
\[
\norm{\lr{\nabla}^{\frac{d-3}{2}} \partial (w_1 w_2 w_3)}_{\mathcal{Y}}
\lesssim \sum _{N_1,N_2, N_3, N_4 \in 2^{\mathbb{N}_0}} N_4^{\frac{d-3}{2}} \norm{P_{N_4} \partial (P_{N_1} w_1 P_{N_2} w_2 P_{N_3} w_3)}_{\mathcal{Y}} .
\]
We observe elementary summation properties: For $2 \le q,r <\infty$
\begin{equation} \label{LP}
\sum _{N \in 2^{\mathbb{N}_0}} \norm{P_N f}_{L_t^q L_x^r}
\lesssim \norm{\lr{\nabla}^{\delta} f}_{L_t^q L_x^r} , \quad
\norm{P_N f}_{L_t^q L_x^r} \lesssim \norm{f}_{L_t^q L_x^r}.
\end{equation}
Indeed, putting $a := \max (q,r)$, from the H\"{o}lder and Minkowski inequalities, and the Littlewood-Paley theory, we obtain
\begin{align*}
\sum _{N \in 2^{\mathbb{N}_0}} \norm{P_N f}_{L_t^q L_x^r}
& \le \left( \sum _{N \in 2^{\mathbb{N}_0}} N^{-\delta} \right) ^{\frac{a}{1-a}} \Norm{\norm{\lr{\nabla}^{\delta} P_N f}_{L_t^q L_x^r}}_{l^a_N}
\le \left( \sum _{N \in 2^{\mathbb{N}_0}} N^{-\delta} \right) ^{\frac{a}{1-a}} \Norm{\norm{\lr{\nabla}^{\delta} P_N f}_{l^a_N}}_{L_t^q L_x^r} \\
& \lesssim \norm{\lr{\nabla}^{\delta} f}_{L_t^q L_x^r} .
\end{align*}
The second inequality follows from the $L^q$ boundedness of the Littlewood-Paley projection.

\noindent
\textbf{Case 1:} $vvv$ case.

This is the deterministic case and the estimate is the same as in \cite{HO}.
But, we repeat it for completeness.
In this case, we need not to use the dyadic decomposition.

Thanks to \ref{d1}, the fractional Leibniz rule (see \cite{CW}), H\"{o}lder's inequality, the Sobolev embedding $W^{\frac{d-2}{2}, \frac{2d}{d-1}} (\R ^d) \hookrightarrow L^{2d}(\R ^d)$, and Proposition \ref{prop:Str} yield
\begin{align*}
\norm{\lr{\nabla}^s \partial (vvv)}_{\mathcal{Y}}
& \lesssim \norm{\lr{\nabla}^s |\nabla |^{\frac{1}{2}} (vvv)}_{L_t^{\frac{4}{3}} L_x^{\frac{2d}{d+1}}}
\lesssim \norm{\lr{\nabla}^s |\nabla |^{\frac{1}{2}} v}_{L_t^4 L_x^{\frac{2d}{d-1}}} \norm{v}_{L_t^4 L_x^{2d}}^2 \\
& \lesssim \norm{\lr{\nabla}^s v}_{\mathcal{X}} \norm{|\nabla |^{\frac{d-2}{2}} u}_{L_t^4 L_x^{\frac{2d}{d-1}}}
\lesssim \norm{v}_{X^s} ^3 .
\end{align*}

\noindent
\textbf{Case 2:} $zzz$ case.

Without loss of generality, we may assume $N_1 \ge N_2 \ge N_3$.

\textbf{Subcase 2-1:} $N_1 \sim N_2 \gtrsim N_3, N_4$.

By H\"{o}lder's inequality, the Sobolev embedding $\dot{W}^{2\delta ,\frac{2d}{d+8\delta}} (\R ^d) \hookrightarrow L^{\frac{2d}{d+4\delta}}(\R ^d)$, and \eqref{LP}, we get
\begin{align*}
\norm{P_{N_4}\partial (P_{N_1} z P_{N_2} z P_{N_3} z)}_{\mathcal{Y}}
& \lesssim N_4 \norm{|\nabla |^{-2\delta} P_{N_4} (P_{N_1} z P_{N_2} z P_{N_3} z)}_{L_t^{\frac{1}{1-\delta}}L_x^{\frac{2d}{d+4\delta}}} \\
& \lesssim N_4 \norm{P_{N_4} (P_{N_1} z P_{N_2} z P_{N_3} z)}_{L_t^{\frac{1}{1-\delta}}L_x^{\frac{2d}{d+8\delta}}} \\
& \lesssim N_4 \norm{P_{N_1} z}_{L_t^{\frac{2}{1-2\delta}} L_x^{\frac{d}{4 \delta}}} \norm{P_{N_2}z}_{L_t^{\frac{2}{1-2\delta}} L_x^{\frac{d}{4 \delta}}} \norm{P_{N_3}z}_{L_t^{\frac{1}{\delta}} L_x^{\frac{2d}{d-8\delta}}} \\
& \lesssim N_1^{-s-1+2\delta} N_2^{-s-1+2\delta} N_3^{-s} N_4 R^3
\end{align*}
for $\omega \in E_R$.
We note that $\frac{d}{4 \delta} > \frac{2d}{d-2+4 \delta}$ holds if $d \ge 3$ and $\delta >0$ is sufficiently small.
We have
\begin{align*}
\sum _{\substack{N_1,N_2,N_3,N_4 \in 2^{\mathbb{N}_0} \\ N_1 \sim N_2 \gtrsim N_3, N_4}} N_4^{\frac{d-3}{2}} \norm{P_{N_4} \partial (P_{N_1} z P_{N_2} z P_{N_3} z)}_{\mathcal{Y}}
& \lesssim \sum _{\substack{N_1,N_2,N_3,N_4 \in 2^{\mathbb{N}_0} \\ N_1 \sim N_2 \gtrsim N_3, N_4}} N_1^{-2s-2+4\delta} N_3^{-s} N_4^{\frac{d-1}{2}} R^3 \\
& \lesssim \sum _{N_1 \in 2^{\mathbb{N}_0}} N_1^{-2s+\max (-s,0) +\frac{d-5}{2}+4\delta} R^3
\lesssim R^3 .
\end{align*}
Here, we used the assumption $s > \max (\frac{d-5}{4}, \frac{d-5}{6})$ and $\delta >0$ is sufficiently small in the last inequality.

\textbf{Subcase 2-2:} $N_1 \sim N_4 \gtrsim N_2, N_3$.

By $d \ge 3$, H\"{o}lder's inequality, and \eqref{LP}, we get
\begin{align*}
\norm{P_{N_4}\partial (P_{N_1} z P_{N_2} z P_{N_3} z)}_{\mathcal{Y}}
& \lesssim \norm{P_{N_4} (P_{N_1} z P_{N_2} z P_{N_3} z)}_{L_t^{2}L_x^{\frac{2d}{d+2}}} \\
& \lesssim \norm{P_{N_1}z}_{L_t^{\frac{2}{1-4\delta}} L_x^{\frac{2d}{d-2}}} \norm{P_{N_2}z}_{L_t^{\frac{1}{\delta}} L_x^{d}} \norm{P_{N_3} z}_{L_t^{\frac{1}{\delta}} L_x^{d}} \\
& \lesssim  N_1^{-s-1+4\delta} N_2^{-s} N_3^{-s} R^3
\end{align*}
for $\omega \in E_R$.
We therefore have
\begin{align*}
\sum _{\substack{N_1,N_2,N_3,N_4 \in 2^{\mathbb{N}_0} \\ N_1 \sim N_4 \gtrsim N_2, N_3}} N_4^{\frac{d-3}{2}} \norm{P_{N_4} \partial (P_{N_1} z P_{N_2} z P_{N_3} z)}_{\mathcal{Y}}
& \lesssim \sum _{\substack{N_1,N_2,N_3,N_4 \in 2^{\mathbb{N}_0} \\ N_1 \sim N_4 \gtrsim N_2, N_3}} N_1^{-s+\frac{d-5}{2}+4\delta} N_3^{-s} N_2^{-s} R^3 \\
& \lesssim \sum _{N_1 \in 2^{\mathbb{N}_0}} N_1 ^{-s+ \max (-2s,0) +\frac{d-5}{2} + 4\delta} R^3 \\
& \lesssim R ^3 ,
\end{align*}
if $s > \max (\frac{d-5}{2}, \frac{d-5}{6})$ and sufficiently small $\delta >0$.

\noindent
\textbf{Case 3:} $vvz$ case.

Without loss of generality, we assume $N_1 \ge N_2$.
We further split the proof into the four subcases.

\textbf{Subcase 3-1:} $N_1 \sim N_2 \gtrsim N_3, N_4$.

If $N_3 \sim 1$, tanks to $P_{N_3} \phi \in H^{\frac{d-3}{2}}(\R ^d)$, it is reduced to Case 1.
We therefore assume $N_3 \gg 1$.
By \ref{d1}, \eqref{LP}, the Sobolev embedding $W^{\frac{d-2}{2}, \frac{2d}{d-1}}(\R ^d) \hookrightarrow L^{2d}(\R ^d)$, we have
\begin{align*}
\norm{P_{N_4}\partial (P_{N_1} v P_{N_2} v P_{N_3} z)}_{\mathcal{Y}}
& \lesssim N_4^{\frac{1}{2}} \norm{P_{N_4} (P_{N_1} v P_{N_2} v P_{N_3} z)}_{L_t^{\frac{4}{3}} L_x^{\frac{2d}{d+1}}} \\
& \lesssim N_4^{\frac{1}{2}} \norm{P_{N_1}v}_{L_t^{4}L_x^{\frac{2d}{d-1}}} \norm{P_{N_2} v}_{L_t^{4} L_x^{2d}} \norm{P_{N_3} z}_{L_t^4 L_x^{2d}} \\
& \lesssim N_4^{\frac{1}{2}} N_2^{\frac{d-2}{2}} \norm{P_{N_1}v}_{L_t^{4}L_x^{\frac{2d}{d-1}}} \norm{P_{N_2} v}_{L_t^{4} L_x^{\frac{2d}{d-1}}} \norm{P_{N_3} z}_{L_t^4 L_x^{2d}} \\
& \lesssim N_1^{-\frac{1}{2}} N_2^{\frac{d-3}{2}} N_3^{-s-\frac{1}{2}} N_4^{\frac{1}{2}} \norm{P_{N_1}v}_{\mathcal{X}} \norm{P_{N_2} v}_{\mathcal{X}} R
\end{align*}
for $\omega \in E_R$.
Hence, from, $s>-1/2$, we obtain
\begin{align*}
& \sum _{\substack{N_1,N_2,N_3,N_4 \in 2^{\mathbb{N}_0} \\ N_1 \sim N_2 \gtrsim N_3, N_4}} N_4^{\frac{d-3}{2}} \norm{P_{N_4} \partial (P_{N_1} v P_{N_2} v P_{N_3} z)}_{\mathcal{Y}} \\
& \qquad \lesssim \sum _{\substack{N_1,N_2,N_3,N_4 \in 2^{\mathbb{N}_0} \\ N_1 \sim N_2 \gtrsim N_3, N_4}} N_1^{\frac{d-4}{2}} N_3^{-s-\frac{1}{2}} N_4^{\frac{d-2}{2}} \norm{P_{N_1} v}_{\mathcal{X}} \norm{P_{N_2} v}_{\mathcal{X}} R \\
& \qquad \lesssim \sum _{\substack{N_1,N_2 \in 2^{\mathbb{N}_0} \\ N_1 \sim N_2}} N_1^{d-3} \norm{P_{N_1} v}_{\mathcal{X}} \norm{P_{N_2} v}_{\mathcal{X}} R \\
& \qquad \lesssim \norm{v}_{X^{\frac{d-3}{2}}}^2 R .
\end{align*}

\textbf{Subcase 3-2:} $N_1 \sim N_3 \gtrsim N_2 , N_4$.

Thanks to $d \ge 3$, \ref{f1}, \ref{f3}, H\"{o}lder's inequality, \eqref{LP}, and the Sobolev embedding $\dot{W}^{\frac{d-2}{2} +2\delta, \frac{2d}{d-4 \delta}} (\R ^d) \hookrightarrow L^{\frac{d}{1-4\delta}} (\R ^d)$, we have
\begin{align*}
\norm{P_{N_4}\partial (P_{N_1} v P_{N_2} v P_{N_3} z)}_{\mathcal{Y}}
& \lesssim N_4 \norm{P_{N_4} (P_{N_1} v P_{N_2} v P_{N_3} z)}_{L_t^{\frac{1}{1-\delta}} L_x^{\frac{2d}{d+8\delta}}} \\
& \lesssim N_4 \norm{P_{N_1} v}_{L_t^{2} L_x^{\frac{2d}{d-2}}} \norm{P_{N_2} v}_{L_t^{\frac{1}{\delta}} L_x^{\frac{d}{1-4\delta}}} \norm{P_{N_3} z}_{L_t^{\frac{2}{1-4\delta}} L_x^{\frac{d}{8 \delta}}} \\
& \lesssim N_4 N_2^{\frac{d-2}{2}} \norm{P_{N_1} v}_{L_t^{2} L_x^{\frac{2d}{d-2}}} \norm{|\nabla |^{2\delta} P_{N_2} v}_{L_t^{\frac{1}{\delta}} L_x^{\frac{2d}{d-4\delta}}} \norm{P_{N_3} z}_{L_t^{\frac{2}{1-4\delta}} L_x^{\frac{d}{8 \delta}}} \\
& \lesssim N_1^{-1} N_2^{\frac{d-2}{2}} N_3^{-s-1+4\delta} N_4 \norm{P_{N_1} v}_{\mathcal{X}} \norm{P_{N_2} v}_{\mathcal{X}} R
\end{align*}
for $\omega \in E_R$.
From \eqref{LP},
\begin{align*}
& \sum _{\substack{N_1,N_2,N_3,N_4 \in 2^{\mathbb{N}_0} \\ N_1 \sim N_3 \gtrsim N_2, N_4}} N_4^{\frac{d-3}{2}} \norm{P_{N_4} \partial (P_{N_1} v P_{N_2} v P_{N_3} z)}_{\mathcal{Y}} \\
& \lesssim \sum _{\substack{N_1,N_2,N_3,N_4 \in 2^{\mathbb{N}_0} \\ N_1 \sim N_3 \gtrsim N_2, N_4}} N_1^{-s-2+4\delta} N_2^{\frac{d-2}{2}} N_4^{\frac{d-1}{2}} \norm{P_{N_1} v}_{\mathcal{X}} \norm{P_{N_2} v}_{\mathcal{X}} R \\
& \lesssim \sum _{N_1 \in 2^{\mathbb{N}_0}} N_1 ^{-s+\frac{d-4}{2}+5 \delta} \norm{P_{N_1}v}_{\mathcal{X}} \norm{v}_{X^{\frac{d-3}{2}}} R
\lesssim \norm{v}_{X^{\frac{d-3}{2}}} ^2 R
\end{align*}
if $s > -\frac{1}{2}$ and $\delta$ is sufficiently small.

\textbf{Subcase 3-3:} $N_1 \sim N_4 \gtrsim N_2, N_3$

From \ref{f1}, \ref{f2}, \ref{f3}, \eqref{LP}, the Sobolev embedding $\dot{W}^{\frac{d-2}{2}+2\delta , \frac{2d}{d-4\delta}}(\R ^d) \hookrightarrow L^{\frac{d}{1-4\delta}}(\R ^d)$, and $s>-\frac{1}{2}$, we have
\begin{align*}
\norm{P_{N_4}\partial (P_{N_1} v P_{N_2} v P_{N_3} z)}_{\mathcal{Y}}
& \lesssim \norm{P_{N_4} (P_{N_1} v P_{N_2} v P_{N_3} z)}_{L_t^{2} L_x^{\frac{2d}{d+2}}} \\
& \lesssim \norm{P_{N_1}v}_{L_t^{\frac{2}{1-4\delta}} L_x^{\frac{2d}{d-2+8\delta}}} \norm{P_{N_2} v}_{L_t^{\frac{1}{\delta}} L_x^{\frac{d}{1-4\delta}}} \norm{P_{N_3} z}_{L_t^{\frac{1}{\delta}} L_x^{d}} \\
& \lesssim N_2^{\frac{d-2}{2}} \norm{P_{N_1}v}_{L_t^{\frac{2}{1-4\delta}} L_x^{\frac{2d}{d-2+8\delta}}} \norm{|\nabla |^{2\delta} P_{N_2} v}_{L_t^{\frac{1}{\delta}} L_x^{\frac{2d}{d-4 \delta}}} \norm{P_{N_3} z}_{L_t^{\frac{1}{\delta}} L_x^{d}} \\
& \lesssim N_1^{-1+4\delta} N_2^{\frac{d-2}{2}}N_3^{-s} \norm{P_{N_1}v}_{\mathcal{X}} \norm{P_{N_2} v}_{\mathcal{X}} R
\end{align*}
for $\omega \in E_R$.
Hence, thanks to \eqref{LP}, we get
\begin{align*}
& \sum _{\substack{N_1,N_2,N_3,N_4 \in 2^{\mathbb{N}_0} \\ N_1 \sim N_4 \gtrsim N_2, N_3}} N_4^{\frac{d-3}{2}} \norm{P_{N_4} \partial (P_{N_1} v P_{N_2} v P_{N_3} z)}_{\mathcal{Y}} \\
& \quad \lesssim \sum _{\substack{N_1,N_2,N_3,N_4 \in 2^{\mathbb{N}_0} \\ N_1 \sim N_4 \gtrsim N_2, N_3}} N_1^{\frac{d-5}{2}+4\delta} N_2^{\frac{d-2}{2}} N_3^{-s} \norm{P_{N_1} v}_{\mathcal{X}} \norm{P_{N_2} v}_{\mathcal{X}} R \\
& \quad \lesssim \sum _{N_1 \in 2^{\mathbb{N}_0}} N_1^{\max (-s,0)+\frac{d-4}{2} + 5\delta} \norm{P_{N_1} v}_{\mathcal{X}} \norm{v}_{X^s} R \\
& \quad \lesssim \norm{v}_{X^{\frac{d-3}{2}}}^2 R
\end{align*}
provided that $s>-\frac{1}{2}$ and $\delta$ is sufficiently small.

\textbf{Subcase 3-4:} $N_3 \sim N_4 \gg N_1 \ge N_2$.

By \ref{f1}, \eqref{LP}, $d \ge 3$, and the Sobolev embedding $\dot{W}^{\frac{d-3}{2}+2\delta , \frac{2d}{d-4\delta}}(\R ^d) \hookrightarrow L^{\frac{2d}{3-8\delta}}(\R ^d)$, we have
\begin{align*}
\norm{P_{N_4}\partial (P_{N_1} v P_{N_2} v P_{N_3} z)}_{\mathcal{Y}}
& \lesssim \norm{P_{N_4} (P_{N_1} v P_{N_2} v P_{N_3} z)}_{L_t^2L_x^{\frac{2d}{d+2}}} \\
& \lesssim \norm{P_{N_1} v}_{L_t^{\frac{1}{\delta}} L_x^{\frac{2d}{3-8\delta}}} \norm{P_{N_2} v}_{L_t^{\frac{1}{\delta}} L_x^{\frac{2d}{3-8\delta}}} \norm{P_{N_3} z}_{L_t^{\frac{2}{1-4\delta}} L_x^{\frac{2d}{d-4+16\delta}}} \\
& \lesssim N_1^{\frac{d-3}{2}} N_2^{\frac{d-3}{2}} \norm{P_{N_1} v}_{L_t^{\frac{1}{\delta}} L_x^{\frac{2d}{d-8\delta}}} \norm{|\nabla |^{2\delta} P_{N_2} v}_{L_t^{\frac{1}{\delta}} L_x^{\frac{2d}{d-4 \delta}}} \norm{P_{N_3} z}_{L_t^{\frac{2}{1-4\delta}} L_x^{\frac{2d}{d-4+16\delta}}} \\
& \lesssim N_1^{\frac{d-3}{2}} N_2^{\frac{d-3}{2}} N_3^{-s-1+4\delta} \norm{P_{N_1}v}_{\mathcal{X}} \norm{P_{N_2} v}_{\mathcal{X}} R
\end{align*}
for $\omega \in E_R$.
From $s > \frac{d-5}{2}$ and \eqref{LP}, we obtain
\begin{align*}
& \sum _{\substack{N_1,N_2,N_3,N_4 \in 2^{\mathbb{N}_0} \\ N_3 \sim N_4 \gtrsim N_1 \ge N_2}} N_4^{\frac{d-3}{2}} \norm{P_{N_4} \partial (P_{N_1} v P_{N_2} v P_{N_3} z)}_{\mathcal{Y}} \\
& \lesssim \sum _{\substack{N_1,N_2,N_3,N_4 \in 2^{\mathbb{N}_0} \\ N_3 \sim N_4 \gtrsim N_1 \ge N_2}} N_1^{\frac{d-3}{2}} N_2^{\frac{d-3}{2}} N_3^{-s+\frac{d-5}{2}+4\delta} \norm{P_{N_1}v}_{\mathcal{X}} \norm{P_{N_2} v}_{\mathcal{X}} R \\
& \lesssim \sum _{N_3 \in 2^{\mathbb{N}_0}} N_3 ^{-s+\frac{d-5}{2}+6\delta} \norm{v}_{X^{\frac{d-3}{2}}}^3 R
\lesssim \norm{v}_{X^{\frac{d-3}{2}}}^2 R ,
\end{align*}
provided that $\delta >0$ is sufficiently small.

\noindent
\textbf{Case 4:} $vzz$ case.

Without loss of generality, we assume $N_2 \ge N_3$.
If $N_2 \sim 1$ and $N_2 \gg N_3 \sim 1$, it is reduced to the Case 1 and Case 3 respectively.
Hence, we also assume $N_3 \gg 1$.
We divide the proof into the four subcases.

\textbf{Subcase 4-1:} $N_1 \sim N_4 \gtrsim N_2 \ge N_3 \gg 1$.

By \ref{f2}, \ref{f3}, and \eqref{LP}, we have
\begin{align*}
\norm{P_{N_4}\partial (P_{N_1} v P_{N_2} z P_{N_3} z)}_{\mathcal{Y}}
& \lesssim \norm{P_{N_4} (P_{N_1} v P_{N_2} z P_{N_3} z)}_{L_t^{2}L_x^{\frac{2d}{d+2}}} \\
& \lesssim \norm{P_{N_1} v}_{L_t^{\frac{2}{1-4\delta}} L_x^{\frac{2d}{d-2+8\delta}}} \norm{P_{N_2} z}_{L_t^{\frac{1}{\delta}} L_x^{\frac{d}{1-2\delta}}} \norm{P_{N_3} z}_{L_t^{\frac{1}{\delta}} L_x^{\frac{d}{1-2\delta}}} \\
& \lesssim N_1^{-1+4\delta} N_2^{-s} N_3^{-s} \norm{P_{N_1}v}_{\mathcal{X}} R^2
\end{align*}
for $\omega \in E_R$.
Thanks to $s>-\frac{1}{2}$ and \eqref{LP}, we have
\begin{align*}
\sum _{\substack{N_1,N_2,N_3,N_4 \in 2^{\mathbb{N}_0} \\N_1 \sim N_4 \gtrsim N_2 \ge N_3 \gg 1}} N_4^{\frac{d-3}{2}} \norm{P_{N_4} \partial (P_{N_1} v P_{N_2} z P_{N_3} z)}_{\mathcal{Y}}
& \lesssim \sum _{\substack{N_1,N_2,N_3,N_4 \in 2^{\mathbb{N}_0} \\ N_1 \sim N_4 \gtrsim N_2 \ge N_3 \gg 1}} N_1^{\frac{d-5}{2}+4\delta} N_2^{-s} N_3^{-s} \norm{P_{N_1}v}_{\mathcal{X}} R^2 \\
& \lesssim \sum _{N_1 \in 2^{\mathbb{N}_0}} N_1 ^{2 \max (-s,0) + \frac{d-5}{2} + 4\delta} \norm{P_{N_1}v}_{\mathcal{X}} R ^2 \\
& \lesssim \norm{v}_{X^{\frac{d-3}{2}}} R ^2
\end{align*}
if $\delta >0$ is sufficiently small.

\textbf{Subcase 4-2:} $N_1 \sim N_2 \gtrsim N_3, N_4$.

By \ref{d1}, \ref{f1}, \eqref{LP}, $d \ge 3$, we have
\begin{align*}
\norm{P_{N_4}\partial (P_{N_1} v P_{N_2} z P_{N_3} z)}_{\mathcal{Y}}
& \lesssim N_4 \norm{P_{N_4} (P_{N_1} v P_{N_2} z P_{N_3} z)}_{L_t^{\frac{1}{1-\delta}} L_x^{\frac{2d}{d+8\delta}}} \\
& \lesssim N_4 \norm{P_{N_1} v}_{L_t^{4} L_x^{\frac{2d}{d-1}}} \norm{P_{N_2} z}_{L_t^{\frac{2}{1-2\delta}} L_x^{4d}} \norm{P_{N_3} z}_{L_t^{4} L_x^{\frac{4d}{1+16\delta}}} \\
& \lesssim N_1^{-\frac{1}{2}} N_2^{-s-1+2\delta} N_3^{-s-\frac{1}{2}} N_4 \norm{P_{N_1}v}_{\mathcal{X}} R^2
\end{align*}
for $\omega \in E_R$.
Thanks to $s>-\frac{1}{2}$, $0<\delta \ll 1$, and \eqref{LP}, we have
\begin{align*}
\sum _{\substack{N_1,N_2,N_3,N_4 \in 2^{\mathbb{N}_0} \\ N_1 \sim N_2 \gtrsim N_3, N_4}} N_4^{\frac{d-3}{2}} \norm{P_{N_4} \partial (P_{N_1} v P_{N_2} z P_{N_3} z)}_{\mathcal{Y}}
& \lesssim \sum _{\substack{N_1,N_2,N_3,N_4 \in 2^{\mathbb{N}_0} \\ N_1 \sim N_2 \gtrsim N_3, N_4}} N_1^{-s-\frac{3}{2}+2\delta} N_3^{-s-\frac{1}{2}} N_4^{\frac{d-1}{2}} \norm{P_{N_1}v}_{\mathcal{X}} R^2 \\
& \lesssim \sum _{N_1 \in 2^{\mathbb{N}_0}} N_1 ^{-s+\frac{d-4}{2}+2\delta} \norm{P_{N_1}v}_{\mathcal{X}} R ^2
\lesssim \norm{v}_{X^{\frac{d-3}{2}}} R ^2 .
\end{align*}

\textbf{Subcase 4-3:} $N_2 \sim N_4 \gtrsim N_1, N_3$.

By \ref{f1}, \ref{f3}, \eqref{LP}, $d \ge 3$, the Sobolev embedding $\dot{W}^{\frac{d-3}{2}+2\delta , \frac{2d}{d-4\delta}}(\R ^d) \hookrightarrow L^{\frac{2d}{3-8\delta}}(\R ^d)$, we have
\begin{align*}
\norm{P_{N_4}\partial (P_{N_1} v P_{N_2} z P_{N_3} z)}_{\mathcal{Y}}
& \lesssim \norm{P_{N_4} (P_{N_1} v P_{N_2} z P_{N_3} z)}_{L_t^{2} L_x^{\frac{2d}{d+2}}} \\
& \lesssim \norm{P_{N_1} v}_{L_t^{\frac{1}{\delta}} L_x^{\frac{2d}{3-8\delta}}} \norm{P_{N_2} z}_{L_t^{\frac{2}{1-4\delta}} L_x^{\frac{4d}{d-1+8\delta}}} \norm{P_{N_3} z}_{L_t^{\frac{1}{\delta}} L_x^{\frac{4d}{d-1+8\delta}}} \\
& \lesssim N_1^{\frac{d-3}{2}} \norm{|\nabla |^{2\delta} P_{N_1} v}_{L_t^{\frac{1}{\delta}} L_x^{\frac{2d}{d-4 \delta}}} \norm{P_{N_2} z}_{L_t^{\frac{2}{1-4\delta}} L_x^{\frac{4d}{d-1+8\delta}}} \norm{P_{N_3} z}_{L_t^{\frac{1}{\delta}} L_x^{\frac{4d}{d-1+8\delta}}} \\
& \lesssim N_1^{\frac{d-3}{2}} N_2^{-s-1+4\delta} N_3^{-s} \norm{P_{N_1}v}_{\mathcal{X}} R^2
\end{align*}
for $\omega \in E_R$.
From \eqref{LP}, we get
\begin{align*}
& \sum _{\substack{N_1,N_2,N_3,N_4 \in 2^{\mathbb{N}_0} \\ N_2 \sim N_4 \gtrsim N_1, N_3}} N_4^{\frac{d-3}{2}} \norm{P_{N_4} \partial (P_{N_1} v P_{N_2} z P_{N_3} z)}_{\mathcal{Y}} \\
& \lesssim \sum _{\substack{N_1,N_2,N_3,N_4 \in 2^{\mathbb{N}_0} \\ N_2 \sim N_4 \gtrsim N_1, N_3}} N_1^{\frac{d-3}{2}} N_2^{-s+\frac{d-5}{2}+4\delta} N_3^{-s} \norm{P_{N_1}v}_{\mathcal{X}} R^2 \\
& \lesssim \sum _{N_2 \in 2^{\mathbb{N}_0}} N_2^{-s+\frac{d-5}{2} + 5\delta +\max (-s,0)} \norm{v}_{X^{\frac{d-3}{2}}} R ^2
\lesssim \norm{v}_{X^{\frac{d-3}{2}}} R ^2,
\end{align*}
provided that $s > \max (\frac{d-5}{2}, \frac{d-5}{4})$ and sufficiently small $\delta >0$.

\textbf{Subcase 4-4:} $N_2 \sim N_3 \gg N_1, N_4$.

By \ref{f1}, \eqref{LP}, and the Sobolev embedding $W^{\frac{d-3}{2}+2\delta , \frac{2d}{d-4\delta}}(\R ^d) \hookrightarrow L^{\frac{2d}{3-8\delta}}(\R ^d)$, we have
\begin{align*}
\norm{P_{N_4}\partial (P_{N_1} v P_{N_2} z P_{N_3} z)}_{\mathcal{Y}}
& \lesssim N_4 \norm{P_{N_4} (P_{N_1} v P_{N_2} z P_{N_3} z)}_{L_t^{\frac{1}{1-\delta}} L_x^{\frac{2d}{d+8\delta}}} \\
& \lesssim N_4 \norm{P_{N_1} v}_{L_t^{\frac{1}{\delta}} L_x^{\frac{2d}{3-8\delta}}} \norm{P_{N_2} z}_{L_t^{\frac{2}{1-2\delta}} L_x^{\frac{4d}{d-3+16\delta}}} \norm{P_{N_3} z}_{L_t^{\frac{2}{1-2\delta}} L_x^{\frac{4d}{d-3+16\delta}}} \\
& \lesssim N_4 N_1^{\frac{d-3}{2}} \norm{|\nabla |^{2\delta} P_{N_1} v}_{L_t^{\frac{1}{\delta}} L_x^{\frac{2d}{d-4\delta}}} \norm{P_{N_2} z}_{L_t^{\frac{2}{1-2\delta}} L_x^{\frac{4d}{d-3+16\delta}}} \norm{P_{N_3} z}_{L_t^{\frac{2}{1-2\delta}} L_x^{\frac{4d}{d-3+16\delta}}} \\
& \lesssim N_1^{\frac{d-3}{2}} N_2^{-s-1+2\delta} N_3^{-s-1+2\delta} N_4 \norm{P_{N_1}v}_{\mathcal{X}} R^2
\end{align*}
for $\omega \in E_R$.
We get by \eqref{LP}
\begin{align*}
& \sum _{\substack{N_1,N_2,N_3,N_4 \in 2^{\mathbb{N}_0} \\ N_2 \sim N_3 \gtrsim N_1, N_4}} N_4^{\frac{d-3}{2}} \norm{P_{N_4} \partial (P_{N_1} v P_{N_2} z P_{N_3} z)}_{\mathcal{Y}}^2 \\
& \lesssim \sum _{\substack{N_1,N_2,N_3,N_4 \in 2^{\mathbb{N}_0} \\ N_2 \sim N_3 \gtrsim N_1, N_4}} N_1^{\frac{d-3}{2}} N_2^{-2s-2+4\delta} N_4^{\frac{d-1}{2}} \norm{P_{N_1}v}_{\mathcal{X}} R^2 \\
& \lesssim \sum _{N_2 \in 2^{\mathbb{N}_0}} N_2^{-2s+\frac{d-5}{2}+5\delta} \norm{v}_{X^{\frac{d-3}{2}}} R ^5
\lesssim \norm{v}_{X^{\frac{d-3}{2}}} R ^2,
\end{align*}
provided that $s > \frac{d-5}{4}$ and sufficiently small $\delta >0$.
\end{proof}

\begin{rmk} \label{rmk:outside1}
From \ref{d1}, \ref{f1}, \ref{f2}, \ref{f3} in the proof of Lemma \ref{lem:nonest}, Lemma \ref{lem:stcStr} implies that $E_R$ in Lemma \ref{lem:nonest} satisfies the bound
\[
P(\Omega \backslash E_R) \le C \exp \left( -c \frac{R^2}{\norm{\phi}_{H^s}^2} \right) .
\]
\end{rmk}

In the local in time case, that is the estimates with $X^{\frac{d-3}{2}}$ replaced by $X^{\frac{d-3}{2}}((-T,T))$, we can gain the factor $T^{\frac{\delta}{4}}$.
Indeed, for example, $\norm{P_{N_3} z}_{L_t^4 L_x^{2d}}$, $\norm{P_{N_3} z}_{L_t^{\frac{2}{1-4\delta}} L_x^{\frac{d}{8\delta}}}$, $\norm{P_{N_3} z}_{L_t^{\frac{1}{\delta}} L_x^d}$, and $\norm{P_{N_3} z}_{L_t^{\frac{2}{1-4\delta}} L_x^{\frac{2d}{d-4+16\delta}}}$ in the Subcases 3-1, 3-2, 3-3, and 3-4 are bounded by
\[
T^{\frac{\delta}{4}} \norm{P_{N_3} z}_{L_t^{\frac{4}{1-\delta}} L_x^{2d}},
T^{\frac{\delta}{2}} \norm{P_{N_3} z}_{L_t^{\frac{2}{1-5\delta}} L_x^{\frac{d}{8\delta}}},
T^{\frac{\delta}{2}} \norm{P_{N_3} z}_{L_t^{\frac{2}{\delta}} L_x^d}, 
T^{\frac{\delta}{2}} \norm{P_{N_3} z}_{L_t^{\frac{2}{1-5\delta}} L_x^{\frac{2d}{d-4+16\delta}}}
\]
respectively.
Since $(\frac{4}{1-\delta}, \frac{2d}{d-1+\delta})$, $(\frac{2}{1-5\delta}, \frac{2d}{d-2+10\delta})$ are Schr\"{o}dinger admissible pairs and $(\frac{2}{\delta}, \frac{2d}{d-4\delta})$ is a biharmonic admissible pair, thanks to Lemma \ref{lem:stcStr}, the above quantities are bounded by
\[
T^{\frac{\delta}{4}} N_3^{-s} R
\]
outside a set of probability $\le C \exp (- c \frac{R^2}{\norm{\phi}_{H^s}^2})$.
A similar procedure is applicable for the cases 2 and 4 because at least one pair of the exponents of the Lebesgue norm for $z$ is not admissible.
Hence, we obtain $T^{\frac{\delta}{4}}$ in the all cases, except for the case 1, in the proof of Lemma \ref{lem:nonest}.
We therefore obtain the following local in time estimate.

\begin{lem} \label{lem:nonestl}
Let $d \ge 3$, $\max ( \frac{d-5}{2}, \frac{d-5}{6}) <s< \frac{d-3}{2}$, and $\delta >0$ be sufficiently small depending only on $d$ and $s$.
Given $\phi \in H^s(\R ^d)$, let $\phi ^{\omega}$ be its randomization defined by \eqref{eq:rand}.
For $R>0$, we put
\[
E_R' := \{ \omega \in \Omega : \norm{\phi ^{\omega}}_{H^s} + \norm{S(t) \phi ^{\omega}}_{S_0'(\R )} \le R \} .
\]
Then, we have
\begin{align}
\norm{\Gamma v}_{X^{\frac{d-3}{2}}((-T,T))} & \le C_1' \left( \norm{v}_{X^{\frac{d-3}{2}}((-T,T))}^3 + T ^{\frac{\delta}{4}} R ^3 \right) , \label{eq:nonestl1} \\
\norm{\Gamma v_1- \Gamma v_2}_{X^{\frac{d-3}{2}}((-T,T))} & \le C_2' \left( \norm{v_1}_{X^{\frac{d-3}{2}} ((-T,T))}^2 + \norm{v_2}_{X^{\frac{d-3}{2}}((-T,T))}^2 + T^{\frac{\delta}{4}} R^2 \right) \norm{v_1-v_2}_{X^{\frac{d-3}{2}}((-T,T))} \label{eq:nonestl2}
\end{align}
for all $v, v_1, v_2 \in X^{\frac{d-3}{2}}((-T,T))$ and $R>0$, outside a set of probability $\le C \exp (- c \frac{R^2}{\norm{\phi}_{H^s}^2})$.
\end{lem}

\begin{rmk} \label{rmk:outside2}
We give a similar remark as in Remark \ref{rmk:outside1}.
Since $(\frac{1}{\delta}, \frac{2d}{d-8\delta})$ and $(\frac{2}{\delta}, \frac{2d}{d-4\delta})$ are biharmonic admissible and $(\frac{4}{1-\delta}, \frac{2d}{d-1+\delta})$, $(\frac{2}{1-3\delta}, \frac{2d}{d-2+6\delta})$, and $(\frac{2}{1-5\delta}, \frac{2d}{d-2+10\delta})$ are  Schr\"{o}dinger admissible, Lemma \ref{lem:stcStr} implies that $E_R'$ in Lemma \ref{lem:nonestl} satisfies the bound
\[
P(\Omega \backslash E_R') \le C \exp \left( -c \frac{R^2}{\norm{\phi}_{H^s}^2} \right) .
\]
\end{rmk}

\section{Proof of Theorems \ref{thm:LWP} ,\ref{thm:WP}} \label{S:4}

To use the contraction argument, we establish the almost sure local in time well-posedness and probabilistic small data global existence and scattering.

\begin{proof}[Proof of Theorem \ref{thm:LWP}]
Let $\eta$ be small enough such that
\begin{equation} \label{cond:eta1}
2 C_1' \eta ^2 \le 1 , \quad 2 C_2' \eta ^2 \le \frac{1}{4},
\end{equation}
where $C_1'$ and $C_2'$ are the constants as in \eqref{eq:nonestl1} and \eqref{eq:nonestl2}.
For any $R>0$, we choose $T=T(R)$ such that
\[
T := \min \left( \frac{\eta}{2 C_1' R^3} , \frac{1}{4C_2' R^2} \right) ^{\frac{4}{\delta}} .
\]
We show that  $\Gamma = \Gamma ^{\omega}$ is a contraction on the ball $B_{\eta}$ defined by
\[
B_{\eta} := \{ u \in X^{\frac{d-3}{2}}((-T,T)) : \norm{u}_{X^{\frac{d-3}{2}}((-T,T))} \le \eta \}
\]
outside a set of probability $ \le C \exp ( - c \frac{1}{T^{\gamma} \norm{\phi}_{H^s}})$ for some $\gamma > 0$, which leads the almost sure local in time well-posedness.
Put
\[
\Omega _T := \{ \omega \in \Omega : \norm{\phi ^{\omega}}_{H^s} + \norm{\lr{\nabla}^s S(t) \phi ^{\omega}}_{S_0'((-T,T))} < R \} ,
\]
where the norm $\norm{\cdot}_{S_0'((-T,T))}$ is defined by \eqref{norm:S_0'} with $I=(-T,T)$.
Lemma \ref{lem:nonestl}, Remark \ref{rmk:outside2} and \eqref{cond:eta1} yield that
\begin{align*}
\norm{\Gamma v}_{X^{\frac{d-3}{2}}((-T,T))} & \le C_1' ( \eta ^3 + T^{\frac{\delta}{4}} R^3)  \le \eta , \\\
\norm{\Gamma v_1 - \Gamma v_2}_{X^{\frac{d-3}{2}}((-T,T))} & \le C_2' ( 2 \eta ^2 + T^{\frac{\delta}{4}} R^2) \norm{v_1-v_2}_{X^{\frac{d-3}{2}}((-T,T))} \le \frac{1}{2} \norm{v_1-v_2}_{X^{\frac{d-3}{2}}((-T,T))}
\end{align*}
for $v , v_1, v_2 \in B_{\eta}$, $\omega \in \Omega \backslash \Omega _T$, and
\[
1-P(\Omega _T) \le C \exp \left( - c\frac{R^2}{\norm{\phi}_{H^s}^2} \right)
\sim C \exp \left( - c\frac{1}{T^{\gamma} \norm{\phi}_{H^s}^2} \right)
\]
for some $\gamma >0$.
Setting
\[
\Sigma := \bigcup _{n \in \mathbb{N}} \Omega _{\frac{1}{n}},
\]
we have $P(\Sigma )=1$.
This completes the proof of Theorem \ref{thm:LWP}.
\end{proof}

\begin{proof}[Proof of Theorem \ref{thm:WP}]
Let $\eta >0$ be sufficiently small such that
\begin{equation} \label{cond:eta2}
2C_1 \eta  ^2 \le 1 ,\quad
3 C_2 \eta ^2 \le \frac{1}{2} ,
\end{equation}
where $C_1$ and $C_2$ are the constants as in \eqref{eq:nonest1} and \eqref{eq:nonest2}.
The probabilistic small data global existence is reduced to show that $\Gamma = \Gamma ^{\omega}$ is a contraction on the ball $B_{R}$ defined by
\[
B_{\eta} := \{ u \in X^{\frac{d-3}{2}}(\R ) : \norm{u}_{X^{\frac{d-3}{2}}} \le \eta \}
\]
outside a set of probability $\le C \exp ( - c \frac{\eta ^2}{\norm{\phi}_{H^s}})$.
Put
\[
\Omega _{\phi} := \{ \omega \in \Omega : \norm{\phi ^{\omega}}_{H^s} + \norm{\lr{\nabla}^s \phi ^{\omega}}_{S_0(\R )} <\eta \} ,
\]
where the norm $\norm{\cdot}_{S_0 (\R )}$ is defined by \eqref{norm:S_0} with $I=\R$.
By Lemma \ref{lem:nonest}, Remark \ref{rmk:outside1}, and \eqref{cond:eta2}, we have
\begin{align}
\norm{\Gamma v}_{X^{\frac{d-3}{2}}} & \le 2 C_1 \eta ^3 \le \eta , \label{nonest3} \\
\norm{\Gamma v_1 - \Gamma v_2}_{X^{\frac{d-3}{2}}} & \le 3 C_2 \eta ^2 \norm{v_1-v_2}_{X^{\frac{d-3}{2}}} \le \frac{1}{2} \norm{v_1-v_2}_{X^{\frac{d-3}{2}}} \label{nonest4}
\end{align}
for $v , v_1, v_2 \in B_{\eta}$, $\omega \in \Omega \backslash \Omega _{\phi}$, and
\[
1-P(\Omega _{\phi}) \le C \exp \left( -c \frac{\eta ^2}{\norm{\phi}_{H^s}^2} \right) .
\]

Next, we prove probabilistic small data scattering.
Fix $\omega \in \Omega _{\phi}$ and let $u=u(\eps , \omega )$ be the global in time solution with $u(0,x) = \phi ^{\omega}(x)$ constructed above.
Set $z^{\omega} (t) = S(t) \phi ^{\omega}$ and $v(t) = u(t)-z(t)$.
To show scattering, we need to prove that there exists $v_+^{\omega} \in H^{\frac{d-3}{2}}(\R ^d)$ such that
\[
S(-t) v(t) = \mp i \int _0^t S(-t') \mathcal{N}(v+ z)(t') dt' \rightarrow v_+^{\omega}
\]
in $H^{\frac{d-3}{2}}(\R ^d)$ as $t \rightarrow \infty$.
Proposition \ref{prop:Str} and Lemma \ref{lem:nonest} implies
\[
\Norm{\int _0^{\infty} S(-t') \mathcal{N}(v+ z)(t') dt'}_{H^{\frac{d-3}{2}}}
\lesssim \norm{\lr{\nabla}^{\frac{d-3}{2}} \mathcal{N}(v+z)}_{\mathcal{Y}}
\le C_1 \left( \norm{v}_{X^{\frac{d-3}{2}}}^3 + \eta ^3 \right)
\]
for $\omega \in \Omega _{\phi}$, which shows that the existence of the limit of $S(-t) v(t)$ in $H^{\frac{d-3}{2}}(\R ^d)$ as $t \rightarrow \infty$.
The negative time direction deduces from the same manner.
\end{proof}

\begin{rmk} \label{almostsureGWP}
To get almost sure small data global well-posedness and scattering, we employ the same argument as in the end of the proof of Theorem \ref{thm:LWP}.
Let $\eta$ be small as above.
Replacing $\phi$ in the proof of Theorem \ref{thm:WP} with $\eps \phi$ for $0< \eps \ll 1$, we define
\[
\Omega _{\eps} := \{ \omega \in \Omega : \eps \norm{\phi ^{\omega}}_{H^s} + \eps \norm{\lr{\nabla}^s \phi ^{\omega}}_{S_0 (\R )}< \eta \} .
\]
Then, Lemma \ref{lem:nonest}, Remark \ref{rmk:outside1}, and \eqref{cond:eta2} yield \eqref{nonest3} and \eqref{nonest4} for $v , v_1, v_2 \in B_{\eta}$, $\omega \in \Omega \backslash \Omega _{\eps}$, and
\[
1-P(\Omega _{\eps}) \le C \exp \left( -c \frac{\eta ^2}{\eps ^2 \norm{\phi}_{H^s}^2} \right) .
\]
Putting
\[
\Sigma := \bigcup _{n \in \mathbb{N}} \Omega _{\frac{1}{n}},
\]
we have $P(\Sigma )=1$.
\end{rmk}

\section{Proof of Theorem \ref{thm:GWP}} \label{S:5}

We consider the scaled equation instead of \eqref{4NLS}:
\begin{equation} \label{dil4NLS}
\left\{
\begin{aligned}
& (i\partial _t + \Delta ^2) u_{\mu} = \pm \partial (|u_{\mu}|^2 u_{\mu}), \\
& u_{\mu}(0,x) = (\phi ^{\omega ,\mu})_{\mu} (x) ,
\end{aligned}
\right.
\end{equation}
where $(\phi ^{\omega ,\mu})_{\mu} (x) = \mu ^{-\frac{3}{2}} \phi ^{\omega ,\mu}( \frac{x}{\mu})$.
We write $(\phi ^{\omega ,\mu})_{\mu}$ as $\phi ^{\omega ,\mu}_{\mu}$ for short.
Putting the linear and nonlinear parts of $u_{\mu}$ by $z_{\mu}(t) = z_{\mu}^{\omega}(t) := S(t) \phi ^{\omega ,\mu}_{\mu}$ and $v_{\mu}(t) = u_{\mu}(t) - S(t) \phi ^{\omega , \mu}_{\mu}$ respectively, we reduce \eqref{dil4NLS} to
\begin{equation}
\left\{
\begin{aligned}
& (i \partial _t + \Delta ^2) v_{\mu} = \pm (|v_{\mu} + z_{\mu}|^2 (v_{\mu} + z_{\mu}), \\
& v_{\mu}(0,x) = 0.
\end{aligned}
\right.
\end{equation}

Define $\Gamma _{\mu}$ by
\[
\Gamma _{\mu} v_{\mu} (t) := \mp i \int _0^t S(t-t') \mathcal{N}(v_{\mu} + z_{\mu}) (t') dt'.
\]
We show that for every $\eps >0$ there exists $\mu _0 = \mu _0 ( \eps , \norm{\phi}_{H^s}) >0$ such that for $\mu \in (0, \mu _0)$ the estimates
\begin{align}
\norm{\Gamma _{\mu} v_{\mu}}_{X^{\frac{d-3}{2}}} & \le C_1 \left( \norm{v_{\mu}}_{X^{\frac{d-3}{2}}}^3 + \eta ^3 \right) , \label{eq:nonestdil1} \\
\norm{\Gamma _{\mu} v_{\mu ,1} - \Gamma v_{\mu ,2}}_{X^{\frac{d-3}{2}}} & \le C_2 \left( \norm{v_{\mu ,1}}_{X^{\frac{d-3}{2}}}^2 + \norm{v_{\mu ,2}}_{X^{\frac{d-3}{2}}}^2 + \eta ^2 \right) \norm{v_{\mu ,1}-v_{\mu ,2}}_{X^{\frac{d-3}{2}}} \label{eq:nonestdil2}
\end{align}
for all $v_{\mu}, v_{\mu ,1}, v_{\mu ,2} \in X^{\frac{d-3}{2}}_0$, outside a set of probability $\le \eps$, where $\eta$ satisfies \eqref{cond:eta2}.
From
\[
\mathcal{F}[\psi (D-n) \phi _{\mu}] (\xi )
= \psi (\xi -n) \mathcal{F}[\phi _{\mu}](\xi )
= \psi (\xi -n) \mu ^{d-\frac{3}{2}} \mathcal{F}[\phi](\mu \xi )
= \mathcal{F}[ (\psi ^{\mu} ( D- \mu n) \phi )_{\mu}] (\xi ),
\]
we get
\begin{equation} \label{scale3}
\phi ^{\omega ,\mu}_{\mu}  = ( \phi ^{\omega ,\mu})_{\mu} = \sum _{n \in \mathbb{Z}^d} g_n(\omega ) \psi (D-n) \phi _{\mu} .
\end{equation}
In addiction, a simple calculation shows
\[
\mathcal{F}_x[( S(t) \phi ^{\omega ,\mu})_{\mu}] (\xi )
= \mu ^{d-\frac{3}{2}} e^{-i \frac{t}{\mu ^4} |\mu \xi |^4} \mathcal{F}[\phi ^{\omega ,\mu}] (\mu \xi)
= e^{-it |\xi |^4} \mathcal{F}[\phi ^{\omega ,\mu}_{\mu}](\xi )
= \mathcal{F}[z_{\mu}](t,\xi ) .
\]

Given $\eta$ as in \eqref{cond:eta2} and $\mu >0$, we define $\Omega _{\mu}$ by
\begin{align*}
\Omega _{\mu} := \{ \omega \in \Omega : \norm{\phi ^{\omega ,\mu}_{\mu}}_{H^s} + \norm{\lr{\nabla}^s S(t) \phi _{\mu}^{\omega ,\mu}}_{S_0(\R )} \le \eta \} ,
\end{align*}
where the norm $\norm{\cdot}_{S_0(\R )}$ is defined by \eqref{norm:S_0} with $I=\R$.
Lemma \ref{lem:stcStr}, \ref{d1}, \ref{f1}, \ref{f2}, \ref{f3}, \eqref{scale}, and \eqref{scale3} yield
\[
1-P(\Omega _{\mu}) \le C \exp \left( - c \frac{\eta ^2}{\norm{\phi _{\mu}}_{H^s}^2} \right)
\le C \exp \left( - c \frac{\eta ^2}{\mu ^{d-3-2 \max (s,0)} \norm{\phi}_{H^s}^2} \right)
\]
for $0< \mu <1$.
By setting
\[
\mu _0 \sim \left( \frac{\eta}{\norm{\phi}_{H^s} (-\log \eps)^{\frac{1}{2}}} \right) ^{\frac{1}{\frac{d-3}{2}-s}} ,
\]
we have
\[
1 - P(\Omega _{\mu} ) < \eps
\]
for $\mu \in (0, \mu _0)$.
The same argument as in the proof of Lemma \ref{lem:nonest}, whose $v$ and $z$ are replaced with $v_{\mu}$ and $z_{\mu}$ respectively, shows that the desired estimates \eqref{eq:nonestdil1} and \eqref{eq:nonestdil2}.
By repeating the proof of Theorem \ref{thm:WP}, we obtain Theorem \ref{thm:GWP}.

\section*{Acknowledgment}

The work of the second author was partially supported by JSPS KAKENHI Grant number 26887017.


\begin{thebibliography}{99}

\bibitem{BOP1}
\'{A}. B\'{e}nyi, T. Oh, and O. Pocovnicu, Wiener randomization on unbounded domains and an application to almost sure well-posedness of NLS, arXiv:1405.7326v2.

\bibitem{BOP2}
\'{A}. B\'{e}nyi, T. Oh, and O. Pocovnicu, On the probabilistic Cauchy theory of the cubic nonlinear Schr\"{o}dinger equation on $\R ^d$, $d \ge 3$, arXiv:1405.7327v2.

\bibitem{BT1}
N. Burq and N. Tzvetkov, Random data Cauchy theory for supercritical wave equations I: local theory, Invent. Math. 173 (2008), no. 3, 449-475 (2008).

\bibitem{BT2}
N. Burq and N. Tzvetkov, Random data Cauchy theory for supercritical wave equations II: a global existence result, Invent. Math. 173 (2008), no. 3, 477-496 (2008).

\bibitem{CCT}
M. Christ, J. Colliander, and T. Tao, Asymptotics, frequency modulation, and low regularity ill-posedness for canonical defocusing equations, Amer. J. Math. 125 (2003), no. 6, 1235-1293.

\bibitem{CW}
F. Christ and M. Weinstein, Dispersion of small amplitude solutions of the generalized Korteweg-de Vries equation, J. Funct. Anal. 100 (1991), no. 1, 87-109. 

\bibitem{Dysthe}
K. Dysthe, Note on a modification to the nonlinear Schr\"{o}dinger equation for application to deep water waves, Proc. R. Soc. Lond. Ser. A 369 (1979), 105-114.

\bibitem{Fukumoto}
Y. Fukumoto, Motion of a curved vortex filament: higher-order asymptotics, in: Proc. IUTAM Symp. Geom. Stat. Turbul., 2001, 211-216.

\bibitem{HHW}
C. Hao, L. Hsiao, and B. Wang, Wellposedness for the fourth order nonlinear Schr\"{o}dinger equations, J. Math. Anal. Appl. 320 (2006), no. 1, 246-265.

\bibitem{HN1}
N. Hayashi and P. Naumkin, Large time asymptotics for the fourth-order nonlinear Schr\"{o}dinger equation, J. Differential Equations 258 (2015), 880-905.

\bibitem{HN2}
N. Hayashi and P. Naumkin, Global existence and asymptotic behavior of solutions to the fourth-order nonlinear Schr\"{o}dinger equation in the critical case, Nonlinear Anal. 116 (2015), 112-131.

\bibitem{HO}
H. Hirayama and M. Okamoto, Well-posedness and scattering for the fourth order nonlinear Schr\"odingier equations at the scaling critical regularity, preprint.

\bibitem{HJ1}
Z. Huo and Y. Jia, The Cauchy problem for the fourth-order nonlinear Schr\"{o}dinger equation related to the vortex filament, J. Differential Equations 214 (2005), no. 1, 1-35.

\bibitem{HJ2}
Z. Huo and Y. Jia, Well-posedness for the fourth-order nonlinear derivative Schr\"{o}dinger equation in higher dimension. J. Math. Pures Appl. (9) 96 (2011), no. 2, 190-206. 

\bibitem{Karpman}
V. Karpman, Stabilization of soliton instabilities by higher order dispersion: fourth-order nonlinear Schr\"{o}dinger-type equations, Phys. Rev. E 53 (2) (1996), 1336-1339.

\bibitem{KS} 
V. Karpman, A. Shagalov, Stability of soliton described by nonlinear Schr\"{o}dinger-type equations with higher-order dispersion, Physica D 144 (2000), 194-210.

\bibitem{LM}
J. L\"{u}hrmann and D. Mendelson, Random data Cauchy theory for nonlinear wave equations of power-type on $\R ^3$, arXiv:1309.1225.

\bibitem{Pau}
B. Pausader, Global well-posedness for energy critical fourth-order Schr\"{o}dinger equations in the radial case, Dyn. Partial Differ. Equ. 4 (2007), no. 3, 197-225. 


\bibitem{Wang}
Y. Wang, Global well-posedness for the generalised fourth-order Schr\"{o}dinger equation, Bull. Aust. Math. Soc. 85 (2012), no. 3, 371-379.
\end{thebibliography}
\end{document}